\documentclass[12pt]{article}
\usepackage{xypic}
\usepackage{epsfig}
\usepackage{tikz}
\usepackage{graphicx}
\usepackage{amsthm}
\usepackage{amssymb}
\usepackage{caption}
\usepackage{subcaption}
\usepackage{color}
\usepackage{hyperref}
\usepackage{amsmath,mathrsfs,amsfonts,verbatim,enumitem,leftidx}

\usepackage{cite}

\newtheorem{defn}{Definition}[section]
\newtheorem{prop}{Proposition}[section]
\newtheorem{thm}{Theorem}[section]
\newtheorem{lem}{Lemma}[section]
\newtheorem{rem}{\bf Remark}[section]

\numberwithin{equation}{section}

\newcommand*{\dif}{\mathop{}\!\mathrm{d}}

\begin{document}

\title{ High pointwise emergence and Katok's conjecture for systems with non-uniform structure
 \footnotetext {* Corresponding author}
  }
\author{Yong Ji$^{1}$, Ercai Chen$^{2}$, Zijie Lin$^{2,*}$\\ 
  \small  1 Department of Mathematics, Ningbo University, Ningbo 315211, Zhejiang, P.R. China\\
  \small   2 School of Mathematical Sciences and Institute of Mathematics, Nanjing Normal University,\\
  \small   Nanjing 210023, Jiangsu, P.R.China\\
      \small    e-mail: imjiyong@126.com, ecchen@njnu.edu.cn, zjlin137@126.com.
}
\date{}
\maketitle

\begin{center}
 \begin{minipage}{120mm}
{\small {\bf Abstract.} 
Recently, Kiriki, Nakano and Soma introduced a concept called pointwise emergence as a new quantitative perspective into the study of non-existence of averages for dynamical systems.
In the present paper, we consider the set of points with high pointwise emergence for systems with non-uniform structure and prove that this set carries full topological pressure. For the proof of this result, we show that such systems have ergodic measures of arbitrary intermediate pressures.
}
\end{minipage}
 \end{center}

\vskip0.5cm

\section{Introductions}
Let $(X,\sigma)$ be a dynamical system in the sense that $X$ is a compact metric space and $\sigma:X\to X$ is a continuous map. Denote by $\mathcal{M}(X)$, $\mathcal{M}_\sigma(X)$, $\mathcal{M}_\sigma^e(X)$ the set of Borel probability measures, $\sigma$-invariant Borel  probability  measures, and ergodic $\sigma$-invariant Borel  probability  measures on $X$ with weak* topology respectively.  The statistical behavior of a point $x\in X$ can be reflected by empirical measures. Namely,
\begin{equation*}
\delta_x^n=\frac{1}{n}\sum_{k=0}^{n-1}\delta_{\sigma^kx},
\end{equation*} 
where $\delta_y$ denotes the Dirac measure at $y\in X$. Consider the sequence $\{ \delta_x^n \}$ and let $V(x)$ be the set of its limit points. Then $V(x)\subset\mathcal{M}_\sigma(X)$ is a nonempty closed connected subset.

The argument on asymptotic behavior of empirical measures is very fundamental in ergodic theory. We say a point $x\in X$ has historic behavior if $V(x)$ is not a single point set. The development in the study of historic sets is similar to those in the study of multifractal decomposition sets. Although the historic set is too small to be considered from the viewpoint of ergodic theory according to Birkhoff ergodic theorem, it turns out that the dynamical structure may be quite complex if we study its qualitative aspects of dynamics. Barreira and Schmeling \cite{Barreira2000} showed that for a class of systems, historic set carries full entropy and Hausdorff dimension. In \cite{Chen2003}, the authors proved that historic set has full topological entropy for systems with specification property. \cite{Dong2018} generalized this result to the systems with shadowing property. See \cite{Bomfim2017, Barreira2018, Takens2008, Thomspon2010, Thompson2012, Zhou2013} for more results of historic behavior in the context of thermodynamic formalism.

Recently Berger \cite{Berger2017} introduced a quantitative viewpoint called emergence into the study of infinitude of averages, and further developed it in \cite{Berger2019} with Bochi. Roughly speaking, the notation of emergence is to quantify how far a system is to be ergodic. It is interesting to investigate whether the dynamics of high emergence are typical, that is, the system whose statistical behavior is poorly described by finitely many probability measures. Kiriki, Nakano and Soma \cite{Kiriki2019} introduced a ``local" version of emergence called pointwise emergence as a new quantitative perspective of $V(x)$.
Given $\epsilon>0$ and $x\in X$, the pointwise emergence $\mathcal{E}_x(\epsilon)$ is defined as the capacity of set $V(x)$, i.e.,
\begin{equation*}
	\mathcal{E}_x(\epsilon)=N(V(x),\epsilon),
\end{equation*}
where $N(V(x),\epsilon)$ denotes the $\epsilon$-covering number of $V(x)$. Call $x$ has high pointwise emergence if 
\begin{equation*}
	\lim_{\epsilon\to0}\frac{\log\mathcal{E}_x(\epsilon)}{-\log\epsilon}=\infty.
\end{equation*}
It was proved that there is a residual subset with high pointwise emergence for topological mixing subshift. A generalized notation ``dynamical pointwise emergence" was studied in \cite{Ji2020}.  In \cite{Nakano2020}, the authors proved that the set of points with high pointwise emergence for topological mixing subshifts of finite type has full topological entropy, full Hausdorff dimension, and full topological pressure for any H\"{o}lder continuous potential. 
Recently, Zelerowicz (\cite{Zel2021}) obtains a lower bound of the Hausdorff dimension of the set of points with high pointwise emergence for $C^{1+\alpha}$ diffeomorphisms. 

A natural question is whether the above results are still valid for more general dynamical systems. To study the existence and uniquencess of measures of maximal entropy and equilibrium states under the weaker specification property, Climenhaga etc. \cite{Climenhaga2012} consider shift spaces with the specification property only holds for ``good words" which is large enough in an appropriate sense. After that a series of remarkable works emerged developing a theory of non-uniform specification which found application for intrinsic ergodicity and equilibrium states of various non-uniformly hyperbolic systems. Based on the works in \cite{Climenhaga2012,Climenhaga2013}, Climenhaga, Thompson and Yamamoto \cite{Climenhaga2017} derived level-$2$ large deviations principle for a broad class of symbolic systems, i.e., the subshifts with non-uniform structure. An important class of shifts with non-uniform structure is given by the $\beta$-shifts, which codes the transformations $x\mapsto\beta x$ (mod $1$) for $\beta>1$, $x\in[0,1]$, and has a deep connection with number theory. Inspired by the research of dimension for subshifts in \cite{Chen1997,Gatzouras1997}, we shall discuss topological pressure for subshifts with non-uniform structure to study the dimension of the set of high pointwise emergence.

Now we present main results of this paper. Let $C(X)$ be the Banach space of all continuous real-valued functions on $X$ equipped with the supremum norm $||\cdot||$. For $f\in C(X)$, set
\begin{equation*}
	P_{\rm inf}(\sigma,f)=\inf\left\{ P_\mu(\sigma,f):\mu\in\mathcal{M}_\sigma(X) \right\},
\end{equation*}
where $P_\mu(\sigma,f)=h_\mu(\sigma)+\int f\dif\mu$ and $h_\mu(\sigma)$ denotes the measure theoretic entropy.
Denote by $P_Z(\sigma,f)$, $P_{\rm top}(\sigma,f)$ the Pein-Pitskel pressure of $Z\subset X$, and topological pressure with respect to $f$ respectively.

\begin{thm}\label{key1}
	Let $X$ be a shift space with $\mathcal{L}=\mathcal{L}(X)$. Suppose that $\mathcal{G}\subset\mathcal{L}$ has $(W)$-specification and $\mathcal{L}$ is edit approachable by $\mathcal{G}$. For any $f\in C(X)$ with $P_{\rm inf}(f)<P_{\rm top}(\sigma,f)$, we have
	\begin{equation*}
		P_{H}(\sigma,f)=P_{\rm top}(\sigma,f),
	\end{equation*}
	where $H=\{ x\in X:\lim\limits_{\varepsilon\to0}\frac{\log\mathcal{E}_x(\varepsilon)}{-\log\varepsilon}=\infty \}$.
\end{thm}
\begin{rem}
	Similar proof can obtain that the set of points with high pointwise emergence carries full topological entropy. That is $h_{\rm top}(\sigma,H)=h_{\rm top}(\sigma,X)$.
\end{rem}

To prove above theorem, we need to consider the ergodic measures of intermediate pressure for systems with non-uniform structure. This question originated from following Katok's conjecture.

\begin{quote}
	{\noindent\bf Katok's Conjecture.} For any $C^2$ diffeomorphism $T$ on a Riemannian manifold $X$, one has
	\begin{equation*}
	\left\{  
	h_{\mu}(T): \mu\ \text{is an ergodic measure for}\ (X,T)
	\right\}\supseteq [0,h_{\rm top}(T)).
	\end{equation*}
\end{quote}
The relative work and progress on this conjecture see \cite{Guan2017, Konieczny2018, Li2018, Li2020, Quas2016, Sun2010, Sun20102, Sun2012, Sun2019, sun2020, Sun2020, Ures2012, Yang2020}. In \cite{Sun2020}, Sun showed that it has dense intermediate pressures and dense intermediate entropies of ergodic measures for systems with Climenhaga-Thompson structures which defined in  \cite{Climenhaga2016}. That is for such system $(X,\sigma)$ and a continuous real valued potential $\phi$, the set $\{ P_\mu(\sigma, \phi):\mu\in\mathcal{M}_f^e(X) \}$ is dense in the interval $[P^*(\phi), P_{\rm top}(\sigma, \phi)]$, where
\begin{equation*}
	P^*(\phi):=\liminf_{n\to\infty}\sup_{x\in X}\left\{
	\frac{1}{n}\sum_{k=0}^{n-1}\phi(\sigma^k(x) ) \right\}.
\end{equation*} 
There is a gap between Climenhaga-Thompson structures and non-uniform structure although they are all defined by the way of gluing orbit properties. We emphasize here that Sun \cite{Sun2019} also proved that the systems with approximate product property and asymptotic entropy expansiveness have ergodic measures of arbitrary intermediate pressures. We cannot directly use this result since there is a type of subshifts with non-uniform structure that does not satisfy approximate product property (see Section 6).

\begin{thm}\label{key2}
	Let $X$ be a shift space with $\mathcal{L}=\mathcal{L}(X)$. Suppose that $\mathcal{G}\subset\mathcal{L}$ has $(W)$-specification and $\mathcal{L}$ is edit approachable by $\mathcal{G}$. 
	Then for $f\in C(X)$ with $P_{\rm inf}(\sigma,f)<P_{\rm top}(\sigma,f)$ and any $\alpha\in(P_{\rm inf}(\sigma,f),P_{\rm top}(\sigma,f)]$, there is $\nu_\alpha\in \mathcal{M}_\sigma^e(X)$ such that $P_{\nu_\alpha}(\sigma,f)=\alpha$.
\end{thm}
\section{Preliminaries}
In this preliminary section, we provide a number of necessary terminologies and recall some lemmas that will be used in the proof of our main result. 
\subsection{Pointwise emergence and infinite dimension simplex}
Let $(X,d)$ be a compact metric space, it is well known that $\mathcal{M}(X)$ is also compact with the weak*-topology. We mainly consider the Wasserstein metric which is compatible with weak*-topology. Given $\mu, \nu\in\mathcal{M}(X)$, define
\begin{equation*}
	W(\mu,\nu)=\sup_{\varphi\in{\rm Lip}^1(X,[-1,1]) }\bigg|\int_X\varphi(x)\dif\mu(x)-\int_X\varphi(x)\dif\nu(x)\bigg|,
\end{equation*}
where ${\rm Lip}^1(X,[-1,1])$ denotes the space of functions on $X$ with values in $[-1,1]$ such that the Lipschitz constant is bounded by $1$. Then the map $x\mapsto\delta_x$ becomes an isometric embedding of $X$ into $\mathcal{M}(X)$.

\begin{defn}{\rm \cite[pointwise emergence]{Kiriki2019}}
	Given $\varepsilon>0$ and $x\in X$, the pointwise emergence $\mathcal{E}_x(\varepsilon)=\mathcal{E}_x(\varepsilon,\sigma)$ of $\sigma$ at scale $\varepsilon$ at $x$ is defined by
	\begin{align*}
		\mathcal{E}_x(\varepsilon)&=\min\bigg\{ N\in\mathbb{N}:\ \text{there eixsts}\ \{ \mu_j\}_{j=1}^N\subset\mathcal{P}(X)\  \text{such that}\\ 
		&\hspace{6cm}\limsup_{n\to\infty}\min_{1\leq j\leq N}W(\delta_x^n,\mu_j)\leq\varepsilon \bigg\}\\
		&=N(V(x),\varepsilon),
	\end{align*}
	where $N(V(x),\varepsilon)$ is the $\varepsilon$-covering number of set $V(x)$ with metric $W$.
\end{defn}

For $L\geq 1$, denote by $A_L$ the set of all $L$-dimensional probability vectors, that is
\begin{equation*}
	A_L=\left\{ {\bf t}=(t_0,t_1,\ldots,t_L)\in[0,1]^{L+1}:\sum_{i=0}^{L}t_i=1 \right\}.
\end{equation*}
For ${\bf n}=(n_0,n_1,\ldots,n_L)\in\mathbb{Z}_+^{L+1}$, define
\begin{equation*}
	\overline{t}({\bf n})=\left( \frac{n_0}{\sum_{i=0}^Ln_i},
	\frac{n_1}{\sum_{i=0}^Ln_i},
	\ldots,
	\frac{n_L}{\sum_{i=0}^Ln_i} \right)\in A_L.
\end{equation*}

Let $\mathcal{T}=\{ \mu^{(l)} \}_{l\geq0}$ be a sequence of Borel probability measures on $X$  and ${\bf t}\in A_L$, set
\begin{equation*}
	\mu_{\bf t}=\mu_{\bf t}(\mathcal{T})=\sum_{i=0}^{L}t_i\mu^{(i)}\in\mathcal{M }(X).
\end{equation*}
Finally, write
\begin{equation*}
	\Delta_L(\mathcal{T})=\left\{ \mu_{\bf t}: {\bf t}\in A_L \right\},\ \ \Delta(\mathcal{T})=\bigcup_{L\geq 1}\Delta_L,
\end{equation*}
and
\begin{equation*}
	E(\mathcal{T})=\left\{ x\in X:\Delta(\mathcal{T}) \subset V(x) \right\}.
\end{equation*}
It was proved in \cite{Kiriki2019} that for any $x\in E(\mathcal{T})$,
	\begin{equation*}
		\lim_{\varepsilon\to0}\frac{\log\mathcal{E}_x(\varepsilon)}{-\log\varepsilon}=\infty.
	\end{equation*}

\subsection{Systems with non-uniform structure}
In this subsection, we recall the systems with non-uniform structure, a class of subshifts which was introduced in \cite{Climenhaga2017} by Climenhaga, Thompson and Yamamoto.

Let $m\geq1$ be an integer and $\mathcal{A}=\{ 1,2,\ldots,m \}$. Consider $\mathcal{A}^\mathbb{N}=\{ (w_i)_{i\geq1}:w_i\in\mathcal{A} \}$ with metric
\begin{equation*}
	d(x,y)=\sum_{j=1}^\infty\frac{|x_j-y_j|}{2^j}.
\end{equation*}  
Let $\sigma$ be the left shift on $\mathcal{A}^\mathbb{N}$ and call $(\mathcal{A}^\mathbb{N},\sigma)$ a full shift. For a nonempty subset $X\subset\mathcal{A}^\mathbb{N}$, we say $(X,\sigma)$ a \textit{subshift} if $X$ is compact and $\sigma(X)\subset X$. 
Denote by $\mathcal{L}_n(X)$ be the set of all $n$-length words that appear in some $x\in X$. Then let $\mathcal{L}=\mathcal{L}(X)=\cup_{n\geq1} \mathcal{L}_n(X)$ and call it the \textit{language} of $X$.
For $u\in\mathcal{L}$, let $|u|$ denote the length of $u$ and $[u]$ denote the cylinder set, that is the set of $x\in X$ beginning with the word $u$. For $x\in X$ and $n\geq 1$, write $[x]_n=[x_1,\ldots,x_n]$.

Let $X$ be a shift space with language $\mathcal{L}$. We say a subset $\mathcal{G}\subset\mathcal{L}$ has \textit{$(W)$-specification with gap length $\tau\in\mathbb{N}$} if for every $u,v\in\mathcal{G}$, there exists $w\in\mathcal{L}$ such that $uwv\in\mathcal{G}$ and $|w|\leq\tau$.

The \textit{edit} of a word $w=w_1\cdots w_n$ is defined by one of the following three actions,\\
(1) substitution: $w\mapsto w'=w_1\cdots w_{i-1}aw_{i+1}\cdots w_n$;\\
(2) insertion: $w\mapsto w'=w_1\cdots w_ia'w_{i+1}\cdots w_n$;\\
(3) deletion: $w\mapsto w'=w_1\cdots w_{i-1}w_{i+1}\cdots w_n$.\\
Here $1\le i\le n$ and $a,a'\in \mathcal{A}$ are arbitrary symbols. Then given $v,w\in\mathcal{L}$, define the \textit{edit distance} between $v$ and $w$ to be the minimum number of edits required to transform the word $v$ into the word $w$, which denoted by $\hat{d}(v,w)$.
The edit distance does not have much effect on statistical behavior, see \cite[Lemma 2.8]{Climenhaga2017} and \cite[Lemma 2.7]{Zhao2018}.

Call a non-decreasing function $g:\mathbb{N}\to\mathbb{N}$ a \textit{mistake function} if $\lim_{n\to\infty}\frac{g(n)}{n}=0$. For $\mathcal{G}\subset\mathcal{L}$, if there is a mistake function $g$ such that for every $w\in\mathcal{L}$, there exists $v\in\mathcal{G}$ with $\hat{d}(w,v)\leq g(|w|)$, we say that \textit{$\mathcal{L}$ is edit approachable by $\mathcal{G}$}.

\subsection{Topological pressure}
Let us recall here the definitions of topological pressure and Pesin-Pitskel topological pressure on the shift space $(X,\sigma)$. 

For $n\in\mathbb{N}$, let $d_n(x,y)=\max\{ d(\sigma^i(x),\sigma^i(y)): 0\leq i\leq n-1 \}$ for any $x, y\in X$. A subset $E\subset X$ is said to be an \textit{$(n,\varepsilon)$ separated set} if for any distinct $x, y\in E$, $d_n(x,y)>\varepsilon$. 
Write $S_nf(x)=\sum_{i=0}^{n-1}f(\sigma^ix)$ for $f\in C(X)$, $n\in\mathbb{N}$ and $x\in X$.

The \textit{topological pressure} of $(X,\sigma)$ with respect to $f\in C(X)$ is given by
\begin{equation*}
	P_{\rm top}(\sigma,f)=\lim_{\varepsilon\to0}\limsup_{n\to\infty}\frac{1}{n}\log P_n(\sigma,f,\varepsilon),
\end{equation*}
where 
\begin{equation*}
	P_n(\sigma,f,\varepsilon)=\sup\left\{ \sum_{x\in E}e^{S_nf(x)}: E\ \text{is an}\ (n,\varepsilon)\ \text{separated set} \right\}.
\end{equation*}
When $f\equiv0$, we obtain \textit{topological entropy} $h_{\rm top}(\sigma,X)$.

Let $Z\subset X$, $f\in C(X)$, $s\geq0$ and $N\in\mathbb{N}$, define
\begin{equation*}
	M(Z,s,N)=\inf_\alpha
	\left\{  \sum_{[u]\in\alpha}\exp\left( -t|u|+\sup_{x\in[u]}S_{|u|}f(x) \right) \right\},
\end{equation*}
where the infimum is taken over all finite or countable collections $\alpha$ of cylinders with length greater than $N$ and covers $Z$ in the sense that $Z\subset\cup_{[u]\in\alpha}[u]$. Clearly, $M(Z,s,N)$ does not decrease as $N$ increases. Define
\begin{equation*}
	M(Z,s)=\lim_{N\to\infty}M(Z,s,N).
\end{equation*}
The \textit{Pesin-Pitskel topological pressure} of $Z$ is the value where $M(Z,s)$ jumps from $\infty$ to $0$, i.e.,
\begin{equation*}
	P_{Z}(\sigma,f)=\inf\left\{ s\geq 0: M(Z,s)=0 \right\}.
\end{equation*}

For $\mu\in\mathcal{M}(X)$, $f\in C(X)$ and $x\in X$, set
\begin{equation*}
	P_\mu(\sigma,f,x)=\liminf_{n\to\infty}\frac{1}{n}\log\left[ e^{S_nf(x)} \mu([x]_n)^{-1} \right].
\end{equation*}
For $\mu\in\mathcal{M}_\sigma^e(X)$, it follows from Birkhoff's ergodic theorem and Brin-Katok's entropy formula \cite{Brin1983} that
\begin{equation*}
	P_{\mu}(\sigma,f,x)=h_{\mu}(\sigma)+\int f\dif\mu
\end{equation*}
holds for $\mu$-almost every $x\in X$. Write $P_\mu(\sigma,f)=h_\mu(\sigma)+\int f\dif\mu$.
The following result shows that the Pesin-Pitskel topological pressure can be determined by measure theoretic pressure.
\begin{thm}{\rm\cite{Tang2015}}\label{bt}
	Let $f$ be a continuous function on $X$,
	$\mu$ be a Borel probability measure on $X$ and $Z\subset X$ be a Borel subset. For $s\in\mathbb{R}$, if $\mu(Z)>0$ and $P_\mu(\sigma,f,x)\geq s$ for all $x\in Z$, then
    $P_{Z}(\sigma,f)\geq s$.
\end{thm}

\section{Intermediate entropies and pressures}

In this section we prove Theorem \ref{key2}. The approach arises from the work in \cite{Sun2019}. For reader's convenience we provide a proof here in our context.

First, we prove the denseness of intermediate entropies, that is,
\begin{thm}\label{t:entropydense}
	Let $X$ be a shift space with $\mathcal{L}=\mathcal{L}(X)$. Suppose that $\mathcal{G}\subset\mathcal{L}$ has $(W)$-specification and $\mathcal{L}$ is edit approachable by $\mathcal{G}$. Then for any $\alpha\in[0,h_{\rm top}(\sigma,X))$, 
	$$\mathcal{M}_e(X,\sigma,\alpha):=\{\mu\in\mathcal{M}^e_\sigma(X):h_{\mu}(\sigma)=\alpha\}$$
	is a residual subset in the compact metric space 
	$$\mathcal{M}^\alpha(X,\sigma):=\{\mu\in\mathcal{M}_\sigma(X):h_{\mu}(\sigma)\ge\alpha\}.$$
\end{thm}

Suppose that $(X,\sigma)$ is a shift space. 
For $\nu\in\mathcal{M}(X)$, $n\in\mathbb{N}$ and $\eta>0$, let 
$$\mathcal{L}_n^{\nu,\eta}=\{\omega\in\mathcal{L}_n:W(\delta_x^n,\nu)<\eta\text{ for any }x\in[\omega]\}.$$



\begin{lem}\label{l:ed}
	Let $X$ be a shift space with $\mathcal{L}=\mathcal{L}(X)$. Suppose that $\mathcal{G}\subset\mathcal{L}$ has $(W)$-specification and $\mathcal{L}$ is edit approachable by $\mathcal{G}$. Then for any $\nu\in\mathcal{M}^e_{\sigma}(X)$, $h\in(0,h_{\nu}(\sigma))$ and $\epsilon,\xi>0$, there exists $\mu\in\mathcal{M}^e_{\sigma}(X)$ with $W(\mu,\nu)<\xi$ such that 
	$$|h_{\mu}(\sigma)-h|<\epsilon.$$
\end{lem}

\begin{proof}
	Fix $\mu\in\mathcal{M}^e_{\sigma}(X)$, $h\in(0,h_{\mu}(\sigma))$ and $\epsilon,\xi>0$.
	By \cite[Lemma 4.7]{Climenhaga2017}, let $\eta=\frac14\xi<\xi$ and then there exists $N$ such that $\#\mathcal{L}_N^{\nu,\eta}\ge e^{Nh}$.
	According to \cite[Lemma 2.7]{Zhao2018}, there exists a sequence of positive number $\{\eta_n\}$ such that for any $x,y\in X$ and $m,n\in\mathbb{N}$,
	\begin{equation}\label{small}
		\hat{d}(x_1x_2\ldots x_n,y_1y_2\ldots y_m)\leq g(n)\Rightarrow W(\delta_x^n,\delta_y^n)\leq\eta_n.
	\end{equation}
	Hence one can choose $N$ large enough such that $\eta_N<\eta$.
	It follows from \cite[Proposition 4.2 and Lemma 4.3]{Climenhaga2017} that there exist a subset $\mathcal{F}\subset\mathcal{L}$ with free concatenation property (for all $u, v\in\mathcal{F}$, we have $uv\in\mathcal{F}$) and a map $\phi: \mathcal{L}\rightarrow \mathcal{F}$ such that  $\hat{d}(\omega,\phi(\omega))<g(|\omega|)$, where $g:\mathbb{N}\to\mathbb{N}$ is a mistake function. Take $\chi<\epsilon$ small enough such that
	\begin{equation*}
	C\chi-\chi\log\chi<\frac14\epsilon.
	\end{equation*}
	We may always assume that $\frac{g(N)}{N}<\chi$. Then for any fixed $v\in\mathcal{F}$, according to \cite[Lemma 2.6]{Climenhaga2017},
	\begin{equation}\label{ed1}
	\#\{ w\in\mathcal{L}_N^{\nu,\eta}:\phi(w)=v\}\leq CN^C( e^{C\chi}e^{-\chi\log\chi} )^{N}<e^{\frac13N\epsilon}.
	\end{equation}
	
	Note that for any $w\in\mathcal{L}_N^{\nu,\eta}$, $N-g(N)\leq |\phi(w)|\leq N+g(N)$.
	For $N-g(N)\leq t\leq N+g(N)$, we denote by $\phi^t(\mathcal{L}_N^{\nu,\eta})\subset\phi(\mathcal{L}_N^{\nu,\eta})$ the collection of words whose length is $t$. Then there exists $t$ such that 
	\begin{equation}\label{ed2}
	\#\phi^t(\mathcal{L}_N^{\nu,\eta})\geq \frac{ \#\{ \phi(w):w\in\mathcal{L}_N^{\nu,\eta}\}}{2g(N)+1}.
	\end{equation} 
	According to pigeonhole principle and (\ref{ed1}, \ref{ed2}), we may assume that $\phi:\mathcal{L}_N^{\nu,\eta}\to\phi(\mathcal{L}_N^{\nu,\eta})$ is a bijection and elements in $\Gamma^*_M:=\phi(\mathcal{L}_N^{\nu,\eta})$ have same length $M$ by dropping some elements of $\mathcal{L}_N^{\nu,\eta}$.
	Hence,
	\begin{equation}\label{ed3}
	\#\Gamma^*_M\geq e^{M\left(h-\frac12\epsilon\right)}.
	\end{equation}
	
	Choose a subset $\Gamma_M\subset\Gamma^*_M$ such that 
	$$e^{M\left(h-\frac12\epsilon\right)}\le \#\Gamma_M< e^{M\left(h+\frac12\epsilon\right)}.$$
	Then let $Y=(\Gamma_M)^\mathbb{N}$ and $\Lambda=\bigcup_{i=0}^{M-1}\sigma^i(Y)$.
	It is clear that $\Lambda$ is compact and $\sigma$-invariant. Moreover,
	$$h_{\rm top}(\sigma,\Lambda)=\lim_{n\rightarrow\infty}\frac{1}{n}\log \#\mathcal{L}_n(\Lambda)\le\lim_{n\rightarrow\infty}\frac{1}{n}\log (\#\Gamma_M)^{\lfloor\frac{n}{M}\rfloor+1}\le h+\frac12\epsilon,$$
	and 
	$$h_{\rm top}(\sigma,\Lambda)=\lim_{n\rightarrow\infty}\frac{1}{n}\log \#\mathcal{L}_n(\Lambda)\ge\lim_{n\rightarrow\infty}\frac{1}{n}\log (\#\Gamma_M)^{\lfloor\frac{n}{M}\rfloor}\ge h-\frac12\epsilon,$$
	where $\lfloor\frac{n}{M}\rfloor$ denotes the largest integer no greater than $\frac{n}{M}$. 
	By the variational principle, there exists a $\mu\in\mathcal{M}^e_\sigma(\Lambda)\subset\mathcal{M}^e_\sigma(X)$, $|h_\mu(\sigma)-h|<\epsilon$. 
	
	There exists $\omega\in\mathcal{L}_N^{\nu,\eta}$ with $\phi(\omega)=v$ for any $v\in\Gamma^*_M$, which follows from the definition of $\Gamma^*_M$. Then by (\ref{small}), for any $x\in[\omega]$ and $y\in[v]$,
	$$W(\delta_y^M,\nu)\le W(\delta_y^M,\delta_x^N)+W(\delta_x^N,\nu)\le \eta_N+\eta<2\eta.$$
	Then if $y\in\Lambda$ is a generic point for $\mu$,
	$$W(\mu,\nu)\le W(\mu,\delta_y^n)+W(\delta_y^n,\nu),$$
	we have $W(\mu,\nu)<\xi$.
\end{proof}

\begin{lem}\label{l:ed2}
	Let $X$ be a shift space with $\mathcal{L}=\mathcal{L}(X)$. Suppose that $\mathcal{G}\subset\mathcal{L}$ has $(W)$-specification and $\mathcal{L}$ is edit approachable by $\mathcal{G}$. For $0\le \alpha<\alpha'<h_{\rm top}(\sigma)$, denote
	$$\mathcal{M}(\alpha,\alpha'):=\{\mu\in\mathcal{M}_\sigma(X):\alpha\le h_{\mu}(\sigma)<\alpha'\}.$$
	Then $\mathcal{M}_e(\alpha,\alpha'):=\mathcal{M}(\alpha,\alpha')\cap\mathcal{M}^e_\sigma(X)$ is dense in $\mathcal{M}^\alpha(X,\sigma)$.
\end{lem}

\begin{proof}
	Fix any $\mu\in\mathcal{M}^\alpha(X,\sigma)$ and $\eta>0$. 
	Let $\mu_1:=\mu+\frac{\eta}{3D}(\mu_{max}-\mu)$ where $D$ is the diameter of $\mathcal{M}(X,\sigma)$ with respect to the Wasserstein metric $W$ and $\mu_{max}$ is an ergodic measure with maximal entropy, that is, $h_{\mu_{max}}(\sigma)=h_{\rm top}(X,\sigma)$. 
	Then 
	$$W(\mu,\mu_1)<\frac{\eta}{3}\text{ and }h_{\mu_1}(\sigma)>\alpha.$$ 
	By \cite[Proposition 3.6]{Climenhaga2017}, there exists $\mu_2\in\mathcal{M}^e(X,\sigma)$ such that 
	$$W(\mu_1,\mu_2)<\frac{\eta}{3}\text{ and }h_{\mu_2}(\sigma)>\alpha.$$
	By Lemma \ref{l:ed}, there exists $\mu_3\in\mathcal{M}^e(X,\sigma)$ such that 
	$$W(\mu_2,\mu_3)<\frac{\eta}{3}\text{ and }\alpha\le h_{\mu_3}(\sigma)<\min\{h_{\mu_2}(\sigma),\alpha'\}.$$
	Therefore, we have $W(\mu,\mu_3)<\eta$ and $\mu_3\in\mathcal{M}_e(\alpha,\alpha')$, which ends the proof.
\end{proof}

\begin{proof}[Proof of Theorem \ref{t:entropydense}]
	Since the map $\mu\mapsto h_{\mu}(\sigma)$ is upper semi-continuous, we have $\mathcal{M}^\alpha(X,\sigma)$ is a compact metric subspace of $\mathcal{M}(X,\sigma)$, which is a Baire space. 
	Since $\mathcal{M}^e_\sigma(X)$ is a $G_\delta$ set in $\mathcal{M}_\sigma(X)$ (\cite[Proposition 5.7]{Denker1976}), then the set $\mathcal{M}^e_\sigma(X)\cap\mathcal{M}^\alpha(X,\sigma)$ is a $G_\delta$ set in $\mathcal{M}^\alpha(X,\sigma)$. 
	By the upper semi-continuity of the entropy map, for any $\alpha<\alpha'$, 
	$$\mathcal{M}(\alpha,\alpha')=\{\mu\in\mathcal{M}_\sigma(X):h_{\mu}(\sigma)<\alpha'\}\cap\mathcal{M}^\alpha(X,\sigma)$$ 
	is relatively open in $\mathcal{M}^\alpha(X,\sigma)$. 
	Then $\mathcal{M}^e(\alpha,\alpha')=\mathcal{M}(\alpha,\alpha')\cap\mathcal{M}^e_\sigma(X)\cap\mathcal{M}^\alpha(X,\sigma)$ is a $G_\delta$ set in $\mathcal{M}^\alpha(X,\sigma)$. 
	Combining with Lemma \ref{l:ed2}, it implies that $\mathcal{M}^e(\alpha,\alpha')$ is a residual set in $\mathcal{M}^\alpha(X,\sigma)$. 
	Then 
	$$\mathcal{M}^e(X,\sigma,\alpha)=\bigcap_{n=1}^\infty\mathcal{M}(\alpha,\alpha+\frac1n)\cap\mathcal{M}^e_\sigma(X)\cap\mathcal{M}^\alpha(X,\sigma)$$
	is also a residual set in $\mathcal{M}^\alpha(X,\sigma)$, which ends the proof.
\end{proof}

Next, we prove the denseness of intermediate pressures.
The variational principle of topological pressure states $P_{ top}(\sigma,f)=\sup\{ P_\mu(\sigma,f):\mu\in\mathcal{M}_\sigma(X) \}$. Since $(X,\sigma)$ is expansive that there exists equilibrium state $\mu_P\in\mathcal{M}_\sigma^e(X)$, $P_{\rm top}(\sigma,f)=P_{\mu_P}(\sigma,f)$. 
We always assume that $h_{\rm top}(\sigma)>0$.

\begin{lem}\label{presslem}
	Let $X$ be a shift space satisfying the condition in Theorem \ref{key2}. Assume that $P_{\rm inf}(\sigma,f)<\alpha<\alpha'<P_{\rm top}(\sigma,f)$. Then the set
	\begin{equation*}
		\mathcal{P}(\alpha,\alpha')=\left\{
		\mu\in\mathcal{M}_\sigma^e(X):\int f\dif\mu\leq\alpha\leq P_{\mu}(\sigma,f)<\alpha'
		\right\}
	\end{equation*}
	is dense in
	\begin{equation*}
		\mathcal{P}^\alpha=\left\{
		\mu\in\mathcal{M}_\sigma(X): \int f\dif\mu\leq\alpha\leq P_{\mu}(\sigma,f)
		\right\}.
	\end{equation*}
\end{lem}
\begin{proof}
	For $\mu_0\in\mathcal{P}^\alpha$ and $\delta_0>0$, we prove the result
	\begin{equation*}
		\mathcal{P}(\alpha,\alpha')\cap B(\mu_0,\delta_0)\neq\emptyset
	\end{equation*}
	in following $5$ cases.
	\\
	{\bf Case $1$.} $\int f\dif\mu_0<\alpha< P_{\mu_0}(\sigma,f).$
	
	Set 
	\begin{equation*}
		\eta=\frac{1}{3}\min\left\{ \alpha-\int f\dif\mu_0, \alpha'-\alpha, P_{\mu_0}(\sigma,f)-\alpha \right\}.
	\end{equation*}
	Since $\mu\mapsto\int f\dif\mu$ is continuous, there exists $\delta_1<\delta_0$ such that for any $m\in B(\mu_0,\delta_0)$,
	\begin{equation*}
		\left| \int f\dif m-\int f\dif\mu_0 \right|<\eta.
	\end{equation*}
	On the other hand, according to Theorem \ref{t:entropydense} there is $\nu\in B(\mu_0,\delta_1)\cap\mathcal{M}^e_\sigma(X)$ such that
	\begin{equation*}
		h_\nu(\sigma)=\min\left\{
		P_{\mu_0}(\sigma,f),\alpha'
		\right\}
		-\int f\dif\mu_0-\eta\in(0,h_{\mu_0}(\sigma)).
	\end{equation*}
	One has
	\begin{align*}
		P_\nu(\sigma,f)=\min\left\{
		P_{\mu_0}(\sigma,f),\alpha'
		\right\}
		-\int f\dif\mu_0&+\int f\dif\nu-\eta>\min\left\{
		P_{\mu_0}(\sigma,f),\alpha'
		\right\}-2\eta>\alpha;\\
		P_\nu(\sigma,f)<\alpha'-\int f\dif\mu_0&-\eta+\int f\dif\nu<\alpha'.
	\end{align*}
	That is $\nu\in B(\mu_0,\delta_0)\cap\mathcal{P}(\alpha,\alpha')$.
	
	{\noindent\bf Case $2$.} $\int f\dif\mu_0=\alpha<P_{\mu_0}(\sigma,f)$.
	
	Take $\nu_0\in\mathcal{M}_\sigma(X)$ such that $P_{\rm inf}(\sigma,f)<P_{\nu_0}(\sigma,f)<\alpha.$  Note that it for any $t\in[0,1]$, it holds that
	\begin{equation*}
		P_{t\mu_0+(1-t)\nu_0}(\sigma,f)=tP_{\mu_0}(\sigma,f)+(1-t)P_{\nu_0}(\sigma,f).
	\end{equation*}
	Hence there exists $\mu\in B(\mu_0,\delta_1)$ with 
	$\delta_1<\delta_0/2$ and $P_{\mu}(\sigma,f)>\alpha>\int f\dif\mu$. Applying Case $1$ we can find
	\begin{equation*}
		\nu\in B(\mu,\delta_1)\cap\mathcal{P}(\alpha,\alpha')\subset B(\mu_0,\delta_0)\cap\mathcal{P}(\alpha,\alpha').
	\end{equation*} 
	
	{\noindent\bf Case $3$.} $\int f\dif\mu_0<\alpha=P_{\mu_0}(\sigma,f)$.
	
	There exists $\mu_1=t_0\mu_P+(1-t_0)\mu_0$, $t_0<\delta_0/2$ such that $\int f\dif\mu_1<\alpha$. At this time one has
	\begin{equation*}
		\mu_1\in B(\mu_0,\delta_0/2),\ \int f\dif\mu_1<\alpha< P_{\mu_1}(\sigma,f).
	\end{equation*}
	Applying Case $1$ again, there exists $\nu$ with 
	\begin{equation*}
		\nu\in B(\mu_1,\delta_0/2)\cap\mathcal{P}(\alpha,\alpha')\subset B(\mu_0,\delta_0)\cap\mathcal{P}(\alpha,\alpha').
	\end{equation*}
	
	{\noindent\bf Case $4$.} $\int f\dif\mu_0=\alpha=P_{\mu_0}(\sigma,f)$ and $\int f\dif\mu_P\leq\alpha$.
	
	Similar argument as Case $2$ there exists
	\begin{equation*}
		\mu_1\in B(\mu_0,\delta_0/2),\ \int f\dif\mu_1\leq\alpha< P_{\mu_1}(\sigma,f).
	\end{equation*}
	Applying Case $1$ (if $\int f\dif\mu_1<\alpha$) or Case $2$ (if $\int f\dif\mu_1=\alpha$), we have $B(\mu_0,\delta_0)\cap\mathcal{P}(\alpha,\alpha')\neq\emptyset$.
	
	{\noindent\bf Case $5$.} $\int f\dif\mu_0=\alpha=P_{\mu_0}(\sigma,f)$ and $\int f\dif\mu_P>\alpha$.
	
	Fix $\nu_0$ as in Case $2$ with $\int f\dif\nu_0\leq P_{\nu_0}(\sigma,f)<\alpha$. Pick $\nu_1\in\mathcal{M}_\sigma(X)$ with $h_{\nu_1}>0$ and $t\in(0,1)$ such that for $\mu_1=t\nu_1+(1-t)\nu_0$,
	\begin{equation*}
	\int f\dif\mu_1<\alpha,\ h_{\mu_1}(\sigma)>0.
	\end{equation*}
	Since $\int f\dif\mu_1<\alpha<\int f\dif\mu_P$, there exists $\mu_2\in\mathcal{M}_\sigma(X)$ with $\int f\dif\mu_2=\alpha$ and $h_{\mu_2}(\sigma)>0$, hence $P_{\mu_2}(\sigma,f)>\alpha$. Consider the collection $\{ k\mu_0+(1-k)\mu_2:0\leq k\leq 1 \}$. Then we can pick $\mu_3$ such that
	\begin{equation*}
		\mu_3\in B(\mu_0,\delta_0/2),\ P_{\mu_3}(\sigma,f)>\alpha=\int f\dif\mu_3.
	\end{equation*}
	Applying Case $2$ again we have $\emptyset\neq B(\mu_3,\delta_0/2)\cap \mathcal{P}(\alpha,\alpha')\subset B(\mu_0,\delta_0)\cap\mathcal{P}(\alpha,\alpha')$.
\end{proof}

Finally, we can prove the following theorem, which implies Theorem \ref{key2}.

\begin{thm}
	Let $X$ be a shift space satisfying the condition in Theorem \ref{key2}. Then for any $\alpha\in\left( P_{\rm inf}(f),P_{\rm top}(\sigma,f) \right)$, the set $\{ \mu\in\mathcal{M}_\sigma^e(X):P_\mu(\sigma,f)=\alpha \}$ is a residual subset in the nonempty compact metric subspace $\mathcal{P}^\alpha$.
\end{thm}
\begin{proof}
	Since $\alpha\in\left( P_{ inf}(f),P_{\rm top}(\sigma,f) \right)$, there exist $\mu_1,\mu_2\in\mathcal{M}_\sigma(X)$ with $P_{\mu_1}(\sigma,f)<\alpha<P_{\mu_2}(\sigma,f)$. Then there exists $t\in(0,1)$ and $\mu=t\mu_1+(1-t)\mu_2$ with $P_{\mu}(\sigma,f)=\alpha$. This ensures that $\mathcal{P}^\alpha$ is nonempty. Note that
	\begin{equation*}
		\mathcal{P}^\alpha=\left\{  \mu\in\mathcal{M}_\sigma(X):\int f\dif\mu\leq\alpha \right\}
		\cap
		\left\{ \mu\in\mathcal{M}_\sigma(X): P_\mu(\sigma,f)\geq\alpha
		\right\}
	\end{equation*}
	is a closed subset since $\mu\mapsto\int f\dif\mu$ is continuous and $\mu\mapsto P_\mu(\sigma,f)$ is upper semi-continuous.
	
	It follows from $\mathcal{M}_\sigma^e(X)\subset\mathcal{M}_\sigma(X)$ is a $G_\delta$ subset that $\mathcal{M}_\sigma^e(X)\cap\mathcal{P}^\alpha$ is a $G_\delta$ subset of $\mathcal{P}^\alpha$. According to Lemma \ref{presslem} $\mathcal{P}(\alpha,\alpha')$ is dense in $\mathcal{P}^\alpha$, $\forall\alpha'\in(\alpha,P_{\rm top}(\sigma,f))$. The set $\mathcal{P}_{\alpha'}=\{ \mu\in\mathcal{M}_\sigma(X): P_\mu(\sigma,f)<\alpha' \}$ is open, then
	\begin{equation*}
		\mathcal{P}(\alpha,\alpha')=\mathcal{M}_\sigma^e(X)\cap\mathcal{P}^\alpha\cap\mathcal{P}_{\alpha'},
	\end{equation*}
	is a residual subset of $\mathcal{P}^\alpha$. So
	\begin{equation*}
		\left\{ \mu\in\mathcal{M}_\sigma^e(X):P_\mu(\sigma,f)=\alpha \right\}=\bigcap_{k\geq1}\mathcal{P}\left(\alpha,\alpha+\frac{1}{k}\right)
	\end{equation*}
	is a residual set.
\end{proof}

\section{Proof of Theorem \ref{key1}}
In this section we prove our main result. The proof consists of finding a Moran subset in the set of points with high pointwise emergence after picking up a sequence of ergodic measures whose measure theoretic pressures are close to topological pressure using Theorem \ref{key2}. 
\subsection{Setting up the itinerary}
We may assume that $P_{\rm top}(\sigma,f)>0$. According to Theorem \ref{key2} for any small enough $\epsilon>0$, we can take a sequence of ergodic invariant measure $\mathcal{T}=\{ \mu^{(l)} \}_{l\geq 0}$ such that $0<P_{\rm top}(\sigma,f)-\epsilon<P_{\mu^{(0)}}(\sigma,f)<P_{\mu^{(1)}}(\sigma,f)<\ldots<P_{\rm top}(\sigma,f)$. For the case of topological entropy, we use Theorem \ref{t:entropydense} to pick the sequence of ergodic measures.

Take a strictly decreasing sequence $\{\epsilon_L\}_{L\geq 1}$ of positive numbers with $\epsilon_L\to0$ and $\epsilon_1<\epsilon$. For $L\geq 1$ and $0\leq l\leq L$, $\tilde{n}\in\mathbb{N}$, define
\begin{align*}
	\Gamma_{L,l}(\tilde{n})=\bigg\{  
	x\in X: \forall n\geq\tilde{n},\ &W(\delta_x^n, \mu^{(l)})<\frac{\epsilon_L}{2}, \bigg|\frac{S_nf(x)}{n}-\int f\dif\mu^{(l)}\bigg|<\frac{\epsilon_L}{2},\\
	&\hspace{1cm} e^{-n\left(h_{\mu^{(l)}}(\sigma)+\epsilon_L\right)}\leq \mu^{(l)}([x]_n)\leq
    e^{-n\left(h_{\mu^{(l)}}(\sigma)-\epsilon_L\right)}
	\bigg\}.
\end{align*}
Since $\mu^{(l)}$ is ergodic, according to Birkhoff's ergodic theorem and Brin-Katok's entropy formula, we have
\begin{equation}\label{to1}
	\lim_{\tilde{n}\to\infty}\mu^{(l)}(\Gamma_{L,l}(\tilde{n}))=1.
\end{equation}

For large enough $n_0$, take a strictly decreasing sequence $\{ \hat{\epsilon}_L \}$ with $\hat{\epsilon}_L\to0$, $\hat{\epsilon}_1<\frac{1}{2}$ and
\begin{equation*}
	\log(1-\hat{\epsilon}_L)\geq -n_0\epsilon_L.
\end{equation*}
According to (\ref{to1}), we can take $\tilde{n}_{L,l}>n_0$ such that
\begin{equation}\label{cedu}
	\mu^{(l)}(\Gamma_{L,l}(\tilde{n}_{L,l}))>1-\hat{\epsilon}_L.
\end{equation} 
Write $\Gamma_{L,l}=\Gamma_{L,l}(\tilde{n}_{L,l})$ for short. 

For $L\geq 1$, take $\{ {\bf t}_{L,j} \}_{j=1}^{J(L)}\subset A_L$ such that $A_L\subset\bigcup_{j=1}^{J(L)} B({\bf t}_{L,j},\frac{\epsilon_L}{L+1})$. Write $\mu_{L,j}=\mu_{\bf t}$ for ${\bf t}={\bf t}_{L,j}$. 
Let $\mathbb{A}_1=\{ (L,j): L\geq 1, 1\leq j\leq J(L) \}$, $\mathbb{A}_2=\{ (L,j,l): L\geq 1, 1\leq j\leq J(L), 0\leq l\leq L \}$, and consider the lexicographic order on both two sets. For $(L,j)\in\mathbb{A}_1$ (or $(L,j,l)\in\mathbb{A}_2$), denote by $(L^\ast,j^\ast)$ (or $(L^\ast,j^\ast,l^\ast)$) its predecessor, and $(L_\ast,j_\ast)$ (or $(L_\ast,j_\ast,l_\ast)$) its successor with respect to the order.

For each $(L,j)\in \mathbb{A}_1$, there exists a finite sequences of positive integers ${\bf n}'_{L,j}=\{  n'_{L,j}(l) \}_{0\leq l\leq L}$ according to \cite[Lemma 5.3]{Kiriki2019}  satisfying
\begin{align*}
n'_{L,j}(l)\geq \tilde{n}_{L,l},\\
|\overline{t}({\bf n}'_{L,j})-{\bf t}_{L,j}|\leq\frac{\epsilon_L}{L+1},
\end{align*}
where $
\overline{t}({\bf n}'_{L,j}):=\left(\frac{n'_{L,j}(0)}{m'_{(L,j)}},\ldots,\frac{n'_{L,j}(L)}{m'_{(L,j)}}\right)\in A_L$, $m'_{(L,j)}:=\sum_{0\leq l\leq L}n'_{L,j}(l)$.


For any $(L,j,l)\in\mathbb{A}_2$ define
\begin{equation*}
\widetilde{\mathcal{W}}(L,j,l)=\left\{ [x]_{n'_{L,j}(l)}:x\in\Gamma_{L,l} \right\}.
\end{equation*}
Note that $n'_{L,j}(l)\geq \tilde{n}_{L,l}$, hence we have
\begin{equation}\label{close}
\max_{0\leq l\leq L}\sup_{x\in\Gamma_{L,l}}W(\delta_x^{n'_{L,j}(l)},\mu^{(l)})<\frac{\epsilon_L}{2}.
\end{equation}
Moreover we have $\mu^{(l)}\left(\bigcup\widetilde{\mathcal{W}}(L,j,l) \right)>1-\hat{\epsilon}_L$. By the definition of $\Gamma_{L,l}$, one has
\begin{equation*}
	\#\widetilde{\mathcal{W}}(L,j,l)\geq(1-\hat{\epsilon}_L)e^{n'_{L,j}(l)\left(h_{\mu^{(l)}}(\sigma)-\epsilon_L\right)}.
\end{equation*}

By the assumption and \cite[Proposition 4.2 and Lemma 4.3]{Climenhaga2017}, there exists $\mathcal{F}$ with free concatenation property and $\mathcal{L}$ is editable approachable by $\mathcal{F}$. That is, there exists a mistake function $g$ and a map $\phi:\mathcal{L}\to\mathcal{F}$ such that $\hat{d}(w,\phi(w))\leq g(|w|)$, $\forall w\in\mathcal{L}$. Take $\chi<\epsilon_L$ small enough such that
\begin{equation*}
	C\chi-\chi\log\chi<\frac{1}{2}\epsilon_L.
\end{equation*}
We may always assume that $\frac{g(n'_{L,j}(l))}{n'_{L,j}(l)}<\chi$. Then for any fixed $v\in\mathcal{L}$, according to \cite[Lemma 2.6]{Climenhaga2017},
\begin{equation}\label{pp1}
	\#\{ w\in\widetilde{\mathcal{W}}(L,j,l):\phi(w)=v \}\leq Cn'_{L,j}(l)^C( e^{C\chi}e^{-\chi\log\chi} )^{n'_{L,j}(l)}<e^{n'_{L,j}(l)\epsilon_L}.
\end{equation}

Note that for any $w\in\widetilde{\mathcal{W}}(L,j,l)$, $n'_{L,j}(l)-g(n'_{L,j}(l))\leq |\phi(w)|\leq n'_{L,j}(l)+g(n'_{L,j}(l))$.
For $n'_{L,j}(l)-g(n'_{L,j}(l))\leq t\leq n'_{L,j}(l)+g(n'_{L,j}(l))$, we denote by $\phi^t(\widetilde{\mathcal{W}}(L,j,l))\subset\phi(\widetilde{\mathcal{W}}(L,j,l))$ the collection of words whose length is $t$. Then there exists $t$ such that 
\begin{equation}\label{pp2}
	\#\phi^t(\widetilde{\mathcal{W}}(L,j,l))\geq \frac{ \#\{ \phi(w):w\in\widetilde{\mathcal{W}}(L,j,l)  \} }{2g(n_{L,j}(l))+1}.
\end{equation} 
According to pigeonhole principle and (\ref{pp1}, \ref{pp2}), we may assume that $\phi:\widetilde{\mathcal{W}}(L,j,l)\to\phi(\widetilde{\mathcal{W}}(L,j,l))$ is a bijection and elements in $\mathcal{W}(L,j,l):=\phi(\widetilde{\mathcal{W}}(L,j,l))$ have same length $n_{L,j}(l)$ by dropping some elements of $\widetilde{\mathcal{W}}(L,j,l)$.
Hence,
\begin{equation}\label{wnumb1}
	\#\mathcal{W}(L,j,l)\geq(1-\hat{\epsilon}_L)e^{n_{L,j}(l)\left(h_{\mu^{(l)}}(\sigma)-3\epsilon_L\right)}\geq e^{n_{L,j}(l)\left(h_{\mu^{(l)}}(\sigma)-4\epsilon_L\right)}.
\end{equation}
Moreover, according to the definition of $\Gamma_{L,l}$, \cite[Lemma 2.8]{Climenhaga2017}, \cite[Lemma 2.7]{Zhao2018} and formula (\ref{close}), we may assume that for any $w\in\mathcal{W}(L,j,l)$ and $y\in[w]$,
\begin{align}
    W(\delta_{y}^{n_{L,j}(l)},\mu^{(l)})&\leq\epsilon_L,\label{ce}\\
	\bigg| \frac{1}{n_{L,j}(l)}S_{n_{L,j}(l)}(y)-\int f\dif\mu^{(l)} \bigg|&<\epsilon_L. \label{jifen}
\end{align}

Write
\begin{equation*}
	m_{(L,j,l)}:=\sum_{0\leq l'\leq l}n_{L,j}(l'),\ m_{(L,j)}:=m_{(L,j,L)}.
\end{equation*}
Set ${\bf n}_{L,j}=\{ n_{L,j}(l) \}_{0\leq l\leq L}$. We may assume that for any $(L,j,l)\in\mathbb{A}_2$,
\begin{align}
	n_{L,j}(l)\geq\tilde{n}_{L,j},\\
	|\overline{t}({\bf n}_{L,j})-{\bf t}_{L,j}|\leq\frac{\epsilon_L}{L+1}.\label{bijin}
\end{align}
The above can be done by adjusting $\tilde{n}_{L,j}$.
Choose a strictly increasing positive integer sequence $\{T_{L,j} \}_{L\geq 1, 1\leq j\leq J(L)}$ such that for any $(L,j)\in\mathbb{A}_1$,
\begin{align}
	m_{(L,j)}<\frac{1}{L}\cdot\sum_{(L',j')<(L,j)}m_{(L',j')}T_{L',j'};\label{shijian}\\
	\sum_{(L',j')<(L,j)}m_{(L',j')}T_{L',j'}<\epsilon_{L}\cdot T_{L,j}m_{(L,j)}.\label{bizhong}
\end{align}
Denote by $\mathbb{A}_3=\{ (L,j,p,l): L\geq 1, 1\leq j\leq J(L), 1\leq p\leq T_{L,j},  0\leq l\leq L \}$ and also consider the lexicographic order.
Write
\begin{align*}
	M_{(L,j,p,l)}&=\sum_{1\leq k<L}\sum_{j=1}^{J(k)}T_{k,j}m_{(k,j)}+\sum_{1\leq j'<j}T_{L,j'}m_{(L,j')}+(p-1)m_{(L,j)}+m_{(L,j,l)},\\
	&M_{(L,j,p)}=M_{(L,j,p,L)},\ \ M_{(L,j)}=M_{(L,j,T_{L,j})},\ \
	M_L=M_{(L,J(L))}.
\end{align*}

\subsection{Constructing Moran set}
For $(L,j,p,l)\in\mathbb{A}_3$,
set
\begin{equation*}
	\mathcal{C}(L,j,p)=\left\{ u(L,j,p)=\left(u(L,j,p,0), u(L,j,p,1),\ldots,u(L,j,p,L)\right)\in\prod_{l=0}^L\mathcal{W}(L,j,l) \right\},
\end{equation*}
\begin{equation*}
	\mathcal{C}(L,j)=\left\{ u(L,j)=\left( u(L,j,1),\ldots, u(L,j,T_{L,j} )\right)\in\prod_{p=1}^{T_{L,j}}\mathcal{C}(L,j,p) \right\},
\end{equation*}
\begin{equation*}
	\mathcal{C}(L)=\left\{ u(L)= \left( u(L,1),\ldots,u(L,J(L)) \right)\in\prod_{j=1}^{J(L)}\mathcal{C}(L,j) \right\}.
\end{equation*}
Define
\begin{align*}
	\mathcal{D}(L,j,p,l)=\bigg\{ \big(  u(1),\ldots,&u(L-1),u(L,1),\ldots,u(L,j-1),u(L,j,1),\\ &\ldots,u(L,j,p-1),u(L,j,p,1),\ldots,u(L,j,p,l) \big) \bigg\}.
\end{align*}
Then for any $u\in\mathcal{D}(L,j,p,l)$, $|u|=M_{(L,j,p,l)}$. Moreover set
\begin{align*}
    D(L,j,p,l)=\bigcup\left\{ [w]:w\in\mathcal{D}(L,j,p,l) \right\}.
\end{align*}
Finally define
\begin{equation*}
    \Lambda_{L,j,p}=D(L,j,p,L),\ \ 
	\Lambda=\bigcap_{(L,j,p)\in\mathbb{A}_2}\Lambda_{L,j,p}.
\end{equation*}

\subsection{$\Lambda\subset E(\mathcal{T})$}

For fixed $\tilde{L}\geq 1$, ${\bf t}\in A_{\tilde{L}}$ and $\varepsilon>0$, take $L\geq \tilde{L}$ such that $\epsilon_L<{\varepsilon}/{5}$. One can find $1\leq j\leq J(L)$ such that
\begin{equation*}
	|{\bf t}-{\bf t}_{L,j}|<\frac{\epsilon_L}{L+1}.
\end{equation*}
Fix such $L$ and $j$. Combining (\ref{bijin}) we have
\begin{equation*}
	|{\bf t}-\overline{t}({\bf n}_{L,j})|\leq \frac{2\epsilon_L}{L+1}.
\end{equation*}
Hence
\begin{equation}\label{tlj}
W\left(\mu_{\bf t},\sum_{l=0}^L\frac{n_{L,j}(l)}{m_{(L,j)}}\mu^{(l)}\right)<\frac{2\varepsilon}{5}.
\end{equation}
Write $t_0=M_{(L,j)}-T_{L,j}m_{(L,j)}$ and $t_1=M_{(L,j)}$.
\begin{lem}
	For any $x\in\Lambda$, it holds that
	\begin{equation*}
		W\left(\delta_x^{t_1},\sum_{l=0}^L\frac{n_{L,j}(l)}{m_{(L,j)}}\mu^{(l)}\right)<\frac{3\varepsilon}{5}.
	\end{equation*}
\end{lem}
\begin{proof}
	According to \cite[Lemma 3.1]{Kiriki2019} we have
	\begin{equation}\label{bizhong2}
		W\left(\delta_{x}^{t_1},\delta_{\sigma^{t_0}(x)}^{t_1-t_0}\right)\leq\frac{2t_0}{t_1}=2\left(1-\frac{T_{L,j}m_{(L,j)}}{M_{(L,j,T_{L,j},L)}}\right).
	\end{equation}
	Now it is sufficiently to estimate $W\left(\delta_{\sigma^{t_0}(x)}^{t_1-t_0},\sum_{l=0}^L\frac{n_{L,j}(l)}{m_{(L,j)}}\mu^{(l)}\right)$. Write $\sigma^{t_0}x|_{T_{L,j}m_{(L,j)}}=\left(u(1,0),\ldots,u(1,L),u(2,0),\ldots,u(2,L)\ldots,u(T_{L,j},L)\right)\in\mathcal{C}(L,j)$. For fixed $1\leq l\leq L$ and $1\leq p\leq T_{L,j}$, according to (\ref{ce}) for any $y_p\in[u(p,l)]$, it holds that $W(\delta_{y_p}^{n_{L,j}(l)},\mu^{(l)})<\epsilon_L$. Hence
	\begin{equation*}
		W\left(\frac{1}{T_{L,j}}\sum_{p=1}^{T_{L,j}}\delta_{y_p}^{n_{L,j}(l)},\mu^{(l)}\right)<\epsilon_L.
	\end{equation*}
	Therefore
	\begin{align*}
		W\bigg(&\delta_{\sigma^{t_0}(x)}^{t_1-t_0},\sum_{l=0}^L\frac{n_{L,j}(l)}{m_{(L,j)}}\mu^{(l)}\bigg)\\
		&=W\left(
		\sum_{l=0}^L\frac{n_{L,j}(l)}{m_{(L,j)}}\left(
		\frac{1}{T_{L,j}}\sum_{p=1}^{T_{L,j}}\delta_{\sigma^{
		t_0+(p-1)m_{(L,j)}+\sum_{l'<l}n_{L,j}(l')	
		} x    }^{n_{L,j}(l)}
		\right),\sum_{l=0}^L\frac{n_{L,j}(l)}{m_{(L,j)}}\mu^{(l)}
		\right)
		<\epsilon_L.
	\end{align*}
	Here we set $\sum_{l'<0}n_{L,j}(l')=0$.
	Combing (\ref{bizhong}, \ref{bizhong2}) we have
	\begin{equation*}
		W(\delta_x^{t_1},\mu_{L,j})<3\epsilon_L<\frac{3\varepsilon}{5}.
	\end{equation*}
\end{proof}

By above lemma and (\ref{tlj}), we conclude that for any $x\in\Lambda$, ${\bf t}\in A_L$  and $\varepsilon>0$, there exists $n$ such that $W(\delta_x^n,\mu_{\bf t})<\varepsilon$. Therefore $\mu_{\bf t}\in V(x)$. Since $L\geq 1$, ${\bf t}$ are arbitrary, we have $\Delta(\mathcal{T})\subset V(x)$ for any $x\in\Lambda$. Hence $\Lambda\subset E(\mathcal{T})$.

\subsection{Constructing measure}

For each $[w]=[u(1),\ldots,u(L^*,j^*,p^*),u(L,j,p,0),u(L,j,p,1),\ldots,u(L,j,p,L)]\in\Lambda_{L,j,p}$, choose one point $x\in[w]$, and let $\mathcal{I}_{L,j,p}$ be the set of all points constructed in this way. According to (\ref{jifen}), it holds that
\begin{align*}
	\bigg| \frac{1}{n_{L,j}(0)}S_{n_{L,j}(0)}f(\sigma^{M_{(L^*,j^*,p^*)}}x)-\int f\dif\mu^{(0)} \bigg|<\epsilon_L,\\
	\bigg| \frac{1}{n_{L,j}(l+1)}S_{n_{L,j}(l+1)}f(\sigma^{M_{(L^*,j^*,p^*)}+m_{(L,j,l)}}x)-\int f\dif\mu^{(l+1)} \bigg|<\epsilon_L,\ l=0,1,\ldots,L-1.
\end{align*}
Hence combining formula (\ref{wnumb1}),
\begin{align*}
	\#\mathcal{I}_{L,j,p}&=\prod_{k=1}^{L-1}\prod_{r=1}^{J(k)}\prod_{l=0}^{k}\#\mathcal{W}(k,r,l)^{T_{k,r}}
	\times
	\prod_{1\leq j'\leq j}\prod_{l=0}^L\#\mathcal{W}(L,j',l)^{T_{L,j'}}\times \prod_{l=0}^L\#\mathcal{W}(L,j,l)^p\\
	&\geq\prod_{k=1}^{L-1}\prod_{r=1}^{J(k)}\prod_{l=0}^{k}e^{n_{k,r}(l)T_{k,r}\left( h_{\mu^{(l)}}(\sigma)-4\epsilon_k \right)  }\times
	\prod_{1\leq j'\leq j}\prod_{l=0}^L e^{n_{L,j'}(l)T_{L,j'}\left( h_{\mu^{(l)}}(\sigma)-4\epsilon_L\right)  }\\
	&\hspace{6.2cm}\times
	\prod_{l=0}^L
	e^{n_{L,j}(l)p\left( h_{\mu^{(l)}}(\sigma)-4\epsilon_L\right)  }\\
	&\geq e^{S_{M_{(L,j,p)}}(x)}e^{ M_{(L,j,p)}\left( P_{\mu^{(0)}}(\sigma,f)-5\epsilon \right) }.
\end{align*}

For $(L,j,p)\in\mathbb{A}_2$ define
\begin{equation*}
	\nu_{L,j,p}=\frac{1}{\#\mathcal{I}_{L,j,p}}\sum_{y\in\mathcal{I}_{L,j,p}}\delta_y.
\end{equation*}
Let $\nu$ be a limit measure for the sequence $\{ \nu_{L,j,p} \}_{(L,j,p)\in\mathbb{A}_2}\subset\mathcal{M}(X)$.
\begin{lem}
	$\nu(\Lambda)=1$.
\end{lem}
\begin{proof}
	Suppose that $\nu=\lim\limits_{k\to\infty}\nu_{L_k,j_k,p_k}$. For fixed $(L_0,j_0,l_0)$ and any $(L,j,p)>(L_0,j_0,l_0)$, $\nu_{L,j,p}(\Lambda_{L_0,j_0,l_0})=1$ since $\Lambda_{L,j,p}\subset\Lambda_{L_0,j_0,l_0}$. Thus $\nu(\Lambda_{L_0,j_0,l_0} )\geq\limsup\limits_{k\to\infty}\nu_{L_k,j_k,p_k}(\Lambda_{L_0,j_0,l_0})=1$. It follows that $\nu(\Lambda)=\lim\limits_{k\to\infty}\nu(\Lambda_{L_k,j_k,l_k})=1$.
\end{proof}

Moreover, it is obviously that for any $x\in\Lambda$ and $(L,j,p)>(L_0,j_0,l_0)$, we have
\begin{equation*}
	\nu_{L,j,p}([x]_{M_{(L_0,j_0,l_0)}})\leq\frac{1}{\#\mathcal{I}_{(L_0,j_0,l_0)}}.
\end{equation*}
Hence for any $n\in\mathbb{N}$, take $(L,j,p)\in\mathbb{A}_2$ such that $M_{(L,j,p)}\leq n< M_{(L_\ast,j_\ast,p_\ast)}$,
\begin{equation*}
	\nu([x]_n)\leq\nu([x]_{ M_{(L,j,p)} })\leq\frac{1}{\#\mathcal{I}_{(L,j,p)}}
	\leq
	e^{-S_{M_{(L,j,p)}}f(x)} 
	e^{-M_{(L,j,p)}\left(P_{\mu^{(0)}}(\sigma)-5\epsilon\right) }.
\end{equation*}	
Then
\begin{align*}
	\frac{\log[e^{S_nf(x)}\nu([x]_n)^{-1}]}{n}
	&\geq
	\frac{ S_nf(x)-S_{M_{(L,j,p)}}f(x)+M_{(L,j,p)}(P_{\mu^{(0)}}(\sigma,f)-5\epsilon)} {M_{(L_\ast,j_\ast,p_\ast)}}
	\\
	&\geq\frac{(M_{(L,j,p)}-M_{(L_\ast,j_\ast,p_\ast)})||f||}{M_{(L_\ast,j_\ast,p_\ast)}}
	+\left( P_{\mu^{(0)}}(\sigma,f)-5\epsilon\right)\frac{M_{(L,j,p)}}{M_{(L_\ast,j_\ast,p_\ast)}}
	\\
	&\geq P_{ \mu^{(0)} }(\sigma,f)-6\epsilon.
\end{align*}
The last inequality derives from the following fact
\begin{equation*}
	\lim_{(L,j,p)\to\infty }\frac{M_{(L,j,p)}}{M_{(L_\ast,j_\ast,p_\ast)}}=1.
\end{equation*}
Thus for any $x\in\Lambda$, it holds that
\begin{equation*}
	\liminf_{n\to\infty}
	\frac{\log[e^{S_nf(x)}\nu([x]_n)^{-1}]}{n}\geq P_{ \mu^{(0)} }(\sigma,f)-6\epsilon.
\end{equation*}
According to Theorem \ref{bt},
$P_{\Lambda}(\sigma,f)\geq P_{ \mu^{(0)} }(\sigma,f)-6\epsilon\geq P_{\rm top}(\sigma,f)-7\epsilon$. Since $\epsilon$ is arbitrary and $\Lambda\subset H$, we have
\begin{equation*}
	P_{H}(\sigma,f)= P_{\rm top}(\sigma,f).
\end{equation*}

\section{Applications}

In this section, we consider the Hausdorff dimension induced by potentials, and two types of subshifts as applications of Theorem \ref{key1}.

\subsection{Hausdorff dimension}
Let $(X,\sigma)$ be a shift space and $\varphi\in C(X)$ be a strict positive function. We use a metric defined by Gatzouras and Peres \cite{Gatzouras1997}:
\begin{equation*}
	d_\varphi(x,y)=
	\begin{cases}
	1 & \text{if}\ x_1\neq y_1,\\
	\exp[ -\sup_{z\in [x\wedge y]}S_{|x\wedge y|}\varphi(z) ] & \text{if}\ x\neq y,\\
	0 & \text{if}\ x=y,
	\end{cases}
\end{equation*}
where $x\wedge y$ denotes the word with the largest length such that $x, y\in[x\wedge y]$. Let ${\rm dim}_\varphi(\cdot)$ denote the Hausdorff dimension of measures or subsets induced by metric $d_\varphi$. In \cite{Gatzouras1997} the authors showed that if $\varphi$ is a strictly positive continuous function, then there exists an ergodic measure $\nu$ such that
\begin{equation*}
	{\rm dim}_\varphi(\nu)={\rm dim}_\varphi(X).
\end{equation*}
Together with the definition of Pesin-Pitskel topological pressure, we readily get that for any $Z\subset X$, ${\rm dim}_\varphi(Z)$ is the unique solution of Bowen equation $P_Z(-s\varphi)=0$. According to Theorem \ref{key1}, we have
\begin{prop}
	Let $X$ be a shift space with $\mathcal{L}=\mathcal{L}(X)$. Suppose that $\mathcal{G}\subset\mathcal{L}$ has $(W)$-specification and $\mathcal{L}$ is edit approachable by $\mathcal{G}$. For any strictly positive $\varphi\in C(X)$ with $P_{\rm inf}(\sigma,\varphi)<P_{\rm top}(\sigma,\varphi)$, we have
	\begin{equation*}
		{\rm dim}_\varphi(H)=\dim_\varphi(X).
	\end{equation*}
\end{prop}

\subsection{Subshifts}
\noindent\emph{\bf $S$-gap shifts.} For an infinite subset $S\subset\{ 0,1,\ldots \}$, an $S$-gap shift $\Sigma_S$ is a subshift of $\{0,1 \}^\mathbb{Z}$ defined by the rule that the number of $0$'s between consecutive $1$'s is an element in $S$. That is, the language of $\Sigma_S$ is
\begin{equation*}
	\{ 0^n10^{n_1}10^{n_2}1\dots 10^{n_k}10^m:\ 1\leq i\leq k, n_i\in S, n,m\in\mathbb{N} \},
\end{equation*}
together with $\{0^n:n\in\mathbb{N} \}$.

\noindent\emph{\bf $\beta$-shifts.} Fix $\beta>1$, write $\alpha=\{0,1,\ldots,\lceil\beta \rceil-1 \}$, where $\lceil\beta \rceil$ denotes the smallest integer larger than $\beta$. Any $x\in[0,1)$ can be written as
\begin{equation*}
	x=\sum_{n=1}^{\infty}\epsilon_n(x,\beta)\beta^{-n}.
\end{equation*}
Denote by $\epsilon(x,\beta)=(\epsilon_1(x,\beta),\epsilon_2(x,\beta),\ldots)$, then
\begin{equation*}
	\Sigma_\beta=\left\{ \omega\in\mathcal{A}^\mathbb{N}:\ \text{there exists some } x\in[0,1)\ \text{such that } \epsilon(x,\beta)=\omega \right\}.
\end{equation*}
In \cite{Climenhaga2017}, the authors showed that all the subshift factors of $S$-gap and $\beta$-shift satisfy the conditions in Theorem \ref{key1}.

\begin{prop}
	Let $(X,\sigma)$ be any $S$-gap shift $(\Sigma_S,\sigma)$ or $\beta$-shift $(\Sigma_\beta,\sigma)$, $f\in C(X)$ with $P_{\rm inf}(\sigma,f)<P_{\rm top}(\sigma,f)$. Then
	\\
	$(1)$ $h_{\rm top}(\sigma,H)=h_{\rm top}(\sigma, X)$.\\
	$(2)$ 	$P_H(\sigma,f)=P_{\rm top}(\sigma,f)$.\\
	$(3)$ ${\rm dim}_\varphi(H)={\rm dim}_\varphi(X)$ for any strictly positive $\varphi\in C(X)$ with $P_{\rm inf}(\sigma,\varphi)<P_{\rm top}(\sigma,\varphi)$.
	$(4)$ $(X,\sigma)$ has ergodic measures of arbitrary intermediate pressures for $f$.
\end{prop}

\section{Appendix: Non-uniform structure does not imply APP}
In this section, we give an example which shows that our results are not included by the results in \cite{Sun2019}.
First, we shows the definition of approximate product property in \cite{Sun2019}.

For a dynamical system $(X,T)$, we call $(X,T)$ has \textit{approximate product property} (or APP for short) if for any $\delta_1,\delta_2,\epsilon>0$, there exists $M(\delta_1,\delta_2,\epsilon)>0$ such that for any $n\ge M(\delta_1,\delta_2,\epsilon)$, any $k\in\mathbb{N}$ and any points $x_1,x_2,\dots,x_k$, there exist a sequence $0=t_0<t_1<\dots<t_k$ with $n\le t_i-t_{i-1}\le (1+\delta_1)n$ for each $i=1,2,\dots,k$ and a point $z\in X$ such that 	$$\#\{0\le j<n:d(T^{t_i+j}z,T^jx_i)>\epsilon\}<\delta_2n$$
for each $i=1,2,\dots,k$.

Next, we will give a class of subshift with non-uniform structure and construct the required example. 
Recall that $\mathcal{A}$ is a finite alphabet. Fix some $M\in\mathbb{N}$, let $\mathcal{A}^M=\mathcal{F}_1\cup\mathcal{F}_2\cup\cdots\cup\mathcal{F}_l$ be a partition of $\mathcal{A}^M$, that is, $\mathcal{F}_i\cap \mathcal{F}_j=\emptyset$ for any $1\le i\neq j\le l$. For each $1\le i\le l$, let $f_i:\mathbb{N}\rightarrow\mathbb{R}$ be an increasing function, called the \emph{frequency function}.
Set $\mathcal{F}=\{(\mathcal{F}_i,f_i)\}_{i=1}^l$, which we call it the \emph{quasi-admissible words}. Define 
$$
\begin{aligned}
X_{\mathcal{F}}=\bigg\{x\in\mathcal{A}^\mathbb{N}:&\text{ for any }n,k\in\mathbb{N}\text{ and }1\le i\le l,\\
&\#\{0\le j<n:x|_{[k+j,k+j+M)}\in\mathcal{F}_i\}\le f_i(n)\bigg\}.
\end{aligned}
$$
If $\mathcal{F}=\{(\mathcal{F}_i,f_i)\}_{i=1}^2$ where $f_1(n)=n$ and $f_2(n)=0$, then $X_\mathcal{F}$ is a shift of finite type where the set of all forbidden words is $\mathcal{F}_2$.

Finally, we give a typical example to show that non-uniform structure does not imply APP.
Set $\mathcal{A}=\{0,1,2\}$ and $\mathcal{A}^2=\mathcal{F}_1\cup\mathcal{F}_2\cup\mathcal{F}_3$ where 
$$\mathcal{F}_1=\{12,21\},\,\mathcal{F}_2=\{00,11,22\}\text{ and } \mathcal{F}_3=\mathcal{A}^2\setminus(\mathcal{F}_1\cup\mathcal{F}_2).$$
Define their corresponding frequency functions
$$f_1(n)=0,f_2(n)=100+\ln n\text{ and }f_3(n)=n.$$

Let $\mathcal{F}=\{(\mathcal{F}_i,f_i)\}_{i=1}^3$ and then
$$
\begin{aligned}
X_{\mathcal{F}}=\bigg\{x\in\{0,1,2\}^\mathbb{N}:
&\text{ for any }n,i\in\mathbb{N},\,x_{i}x_{i+1}\notin\{12,21\},\\
&\#\{0\le j< n:x_{i+j}x_{i+j+1}=00,\,11\text{ or }22\}\ge 100+\ln n
\bigg\}.
\end{aligned}
$$

Set $$\mathcal{G}=\bigcup_{n\in\mathbb{N}}\bigg\{0a_10a_2\cdots0a_n:a_i\in\{1,2\},\,i=1,\dots,n\bigg\}.$$
It is clear that $\mathcal{G}\subset\mathcal{L}(X_{\mathcal{F}})$ and $\mathcal{G}$ has $(W)$-specification. Since $\frac{f_2(n)}{n}$ converges to $0$, $f_2$ is a mistake function and $\mathcal{L}(X_{\mathcal{F}})$ is edit approachable by $\mathcal{G}$. 

\begin{prop}
	Subshift $(X_{\mathcal{F}},\sigma)$ does not have APP.
\end{prop}
\begin{proof}
	First, we choose some word in $X_{\mathcal{F}}$.
	Fix any $\eta,\delta>0$. 
	For any $n\in\mathbb{N}$ large enough such that there exist $p,r\in\mathbb{N}$ with 
	\begin{equation}\label{e:base}
	\eta\lfloor\ln n\rfloor-1>100+\ln 7,\,\,n=(2p+1)\lfloor\ln n\rfloor+2r,\,2<\ln(2p+3)\text{ and }2r<3\lfloor\ln n\rfloor.
	\end{equation}
	Then we have
	\begin{equation}\label{e:base2}
	p>2\text{ and }n<4p\lfloor\ln n\rfloor.
	\end{equation}
	For a word $W$, denote by $W^{(k)}$ be the word concatenated by $k$ words $W$, that is, 
	$$W^{(k)}=\underbrace{W\cdots W}_{k\text{ times}}.$$
	
	Let $m=\lfloor\ln n\rfloor$ and $\omega=(1(10)^{(p)})^{(m)}(10)^{(r)}$. 
	We claim that $[\omega]\cap X_{\mathcal{F}}\neq\emptyset$. Indeed, for any subword $u$ of $\omega$ such that $11$ appears in $u$ for $i+1$ times, $(1(10)^{(p)})^{(i)}11$ must appear in $u$, that is, $|u|\ge i(2p+1)+2$  And $|\omega|=n$. Noticing that $\frac{\ln(i(2p+1)+2)}{i+1}$ is decreasing while $i$ increases, we have 
	$$\frac{\ln|u|}{i+1}\ge\frac{\ln(i(2p+1)+2)}{i+1}\ge \frac{\ln(m(2p+1)+2)}{m+1} \ge \frac{n-2m}{m+1}\ge 1.$$
	This shows that $[\omega]\cap X_{\mathcal{F}}\neq\emptyset$.
	
	Finally, we prove the proposition by contradiction. 
	Suppose that $(X_{\mathcal{F}},\sigma)$ has approximate product property. 
	Choose $\epsilon>0$ such that $d(x,y)>\epsilon$ if and only if $x_0\neq y_0$ for any $x,y\in X_{\mathcal{F}}$. Then for $n\ge M(1,\delta,\epsilon)$ with (\ref{e:base}) holds, there exist some point $x\in X_{\mathcal{F}}$ and $u^1,u^2,u^3,u^4,v,v',v''$ with $|u^1|=|u^2|=|u^3|=|u^4|=n$ and $|v|,|v'|,|v''|<n$ such that $x|_{[0,4n+|v|+|v'|+|v''|)}=u^1vu^2v'u^3v''u^4$ and 
	$$\#\{0\le j<n: u^i_j\neq\omega_j\}<\delta n\text{ for }i=1,2,3,4.$$
	Since $x\in X_{\mathcal{F}}$, then we have 
	$$
	\begin{aligned}
	\#\{0\le j<&4n+|v|+|v'|+|v''|:x_jx_{j+1}=00,\,11\text{ or }22\}\\
	<&100+\ln(4n+|v|+|v'|+|v''|)<100+\ln(7n).
	\end{aligned}
	$$
	Choose $u^i$ such that 
	\begin{equation}\label{e:upperbound}
	\#\{0\le j<n-1:u^i_ju^i_{j+1}=00,\,11\text{ or }22\}\le\frac14(100+\ln(7n))\le \frac14(1+\eta)m.
	\end{equation}
	Set $u^i=v^1v^2\cdots v^m\tau$, $|v^1|=|v^2|=\cdots=|v^m|=2p+1$ and $|\tau|=2r$.
	Let 
	$$\mathcal{W}=\{0a_10\cdots a_p0: a_i\in\{1,2\},i=1,2,\dots,p\}\text{ and }$$
	$$\mathcal{W}'=\{a_00a_1\cdots 0a_p: a_i\in\{1,2\},i=0,1,\dots,p\}.$$ 
	Note that the length of each word in $\mathcal{W}$ and $\mathcal{W}'$ is $2p+1$. Then
	$$\#\{1\le j \le m:v^j\in\mathcal{W}\cup\mathcal{W}'\}\ge m-\frac14(100+\ln(7n))\ge m-\frac14(1+\eta)m,$$
	where the last inequality holds since (\ref{e:base}).
	On the other hand, since $\#\{0\le j<n: u^i_j\neq\omega_j\}<\delta n$, by (\ref{e:base2}), we have
	$$\#\{1\le j \le m:v^j\in\mathcal{W}'\}<\frac1{2p}\cdot\delta n<2\delta m,$$
	which implies that 
	$$\#\{1\le j \le m:v^j\in\mathcal{W}\}\ge (1-\frac14(1+\eta)-2\delta)m.$$
	Then 
	$$\#\{1\le j < m:v^j,v^{j+1}\in\mathcal{W}\}\ge 2(1-\frac14(1+\eta)-2\delta)m-m.$$
	By (\ref{e:upperbound}), 
	$$
	\begin{aligned}
	\frac14(1+\eta)m\ge&\#\{0\le j<n-1:u^i_ju^i_{j+1}=00\}\\
	\ge&\#\{1\le j < m:v^j,v^{j+1}\in\mathcal{W}\}\\
	\ge&2(1-\frac14(1+\eta)-2\delta)m-m\\
	=&(\frac12-\frac\eta2-2\delta)m,
	\end{aligned}
	$$
	which leads to a contradiction by the arbitrariness of $\eta,\delta$. So $(X_{\mathcal{F}},\sigma)$ does not have APP.
\end{proof}

\noindent\textbf{Acknowledgments.} Both authors were supported by NNSF of China (12071222), the first author was also supported by NNSF of China (12101340).


\begin{thebibliography}{50}
	\bibitem{Bomfim2017} T. Bomfim, P. Varandas, Multifractal analysis for weak Gibbs measures: from large deviations to irregular sets, \emph{Ergod. Th. Dynam. Syst.} {\bf 37}(2017), 79-102.
	
	
	\bibitem{Barreira2018} L. Barreira, J. Li, C. Valls, Topological entropy of irregular sets, \emph{Rev. Mat. Iberoam. } {\bf 34}(2018), 853-878.
	
	
	\bibitem{Barreira2000} L. Barreira, J. Schmeling, Sets of ``non-typical" points have full topological entropy and full Hausdorff dimension, \emph{Israel J. Math.} {\bf 116}(2000), 29-70.
	
	\bibitem{Berger2017} P. Berger, Emergence and non-typicality of the finiteness of the attractors in many topologies, \emph{Proc. Steklov Inst. Math.} {\bf 297}(2017), 1-27.
	
	\bibitem{Berger2019} P. Berger, J. Bochi, On emergence and complexity of ergodic decompositions, \emph{Adv. Math.} {\bf 390}(2021), 107904, 52pp.
	
	
	\bibitem{Brin1983} M. Brin, A. Katok, On local entropy, \emph{Geometric Dynamics (Rio de Janeiro, 1981) (Lecture Notes in Mathematics, 1007).} Springer, Berlin, 1983, pp. 30-38.
	
	\bibitem{Chen2003} E. Chen, K. Tassilo, L. Shu, Topological entropy for divergence points, \emph{Ergod. Th. Dynam. Sys.} {\bf 25}(2005), 1173-1208.
	
	\bibitem{Chen1997} E. Chen, J. Xiong, Dimension and measure theoretic entropy of a subshift in symbolic, \emph{Chinese Sci. Bull.} {\bf 42}(1997), 1193-1196.
	
	\bibitem{Climenhaga2012} V. Climenhaga, D. Thompson, Intrinsic ergodicity beyond specification: $\beta$-shifts, $S$-gap shifts, and their factors, \emph{Israel J. Math.} {\bf 192}(2012), 785-817.
	
	\bibitem{Climenhaga2013} V. Climenhaga, D. Thompson, Equilibrium states beyond specification and the Bowen property, \emph{J. Lond. Math. Soc.} {\bf 87}(2013), 401-427.
	
	\bibitem{Climenhaga2016} V. Climenhaga, D. Thompson, Unique equilibrium states for flows and homeomorphisms with non-uniform structure, \emph{Adv. Math.} {\bf 303}(2016), 744-799.
	
	
	\bibitem{Climenhaga2017} V. Climenhaga, D. Thompson, K. Yamamoto, Large deviations for systems with non-uniform structure, \emph{Trans. Amer. Math. Soc.} {\bf 369}(2017), 4167-4192.
	
	\bibitem{Denker1976} M. Denker, C. Grillenberger, K. Sigmund, \emph{Ergodic theory on compact spaces, } Lecture Notes in Mathematics, Springer-Verlag, Berlin-New York, 1976.
	
	
	\bibitem{Dong2018} Y. Dong, P. Oprocha, X. Tian, On the irregular points for systems with the shadowing property, \emph{Ergod. Th. Dynam. Sys.} {\bf 38}(2018), 2108-2131.
	
	\bibitem{Gatzouras1997} D. Gatzouras, Y. Peres, Invariant measures of full dimension for some expanding maps, \emph{Ergod. Th. Dynam. Syst.} {\bf 17}(1997), 147-167.
	
	\bibitem{Guan2017} L. Guan, P. Sun, W. Wu, Measures of intermediate entropies and homogeneous dynamics, \emph{Nonlinearity}, {\bf 30}(2017), 3349-3361.
	
	\bibitem{Ji2020} Y. Ji, E. Chen, X. Zhou, Entropy and emergence of topological dynamical systems, arXiv: 2005.01548.
	
	\bibitem{Kiriki2019} S. Kiriki, Y. Nakano, T. Soma, Emergence via non-existence of averages, arXiv: 1904.03424.
	
	\bibitem{Konieczny2018} J. Konieczny, M. Kupsa, D. Kwietniak, Arcwise connectedness of the set of ergodic measures of hereditary shifts, \emph{Proc. Amer. Math. Soc.} {\bf 146}(2018), 3425-3438.
	
	\bibitem{Li2018} J. Li, P. Oprocha, Properties of invariant measures in dynamical systems with the shadowing property, \emph{Ergod. Th. Dynam. Syst.} {\bf 38}(2018), 2257-2294.
	
	\bibitem{Li2020} M. Li, Y. Shi, S. Wang, X. Wang, Measures of intermediate entropies for star vector fields, \emph{Israel J. Math.} {\bf 240}(2020), 791-819.

	\bibitem{Nakano2020} Y. Nakano, A. Zelerowicz, Highly irregular orbits for subshifts of finite type: large intersections and emergence, \emph{Nonlinearity} {\bf 34}(2021), 7609-7632.
	
	
	\bibitem{Quas2016} A. Quas, T. Soo, Ergodic universality of some topological dynamical systems, \emph{Trans. Amer. Math. Soc.} {\bf 368}(2016), 4137-4170.
	
	\bibitem{Sun2010} P. Sun, Zero-entropy invariant measures for skew product diffeomorphisms, \emph{Ergod. Th. Dynam. Syst.} {\bf 30}(2010), 923-930.
	
	
	\bibitem{Sun20102} P. Sun, Measures of intermediate entropies for skew product diffeomorphisms, \emph{Discrete Contin. Dyn. Syst.} {\bf 27}(2010), 1219-1230.
	
	\bibitem{Sun2012} P. Sun, Density of metric entropies for linear toral automorphisms, \emph{Dyn. Syst.} {\bf 27}(2012), 197-204.
	
	\bibitem{Sun2019} P. Sun, Ergodic measures of intermediate entropies for dynamical systems with approximate product property, arXiv: 1906.09862.
	
	\bibitem{sun2020} P. Sun, Zero-entropy dynamical systems with the gluing orbit property, \emph{Adv. Math.} {\bf 372}(2020), 107294, 24 pp.
	
	\bibitem{Sun2020} P. Sun, Denseness of intermediate pressures for systems with Climenhaga-Thompson structures, \emph{J. Math. Anal. Appl.} {\bf 487}(2020), 124027, 17 pp.
	
	
	\bibitem{Takens2008} F. Takens, Orbits with historic behaviour, or non-existence of averages, \emph{Nonlinearity} {\bf 21}(2008), T33-T36.
	
	
	\bibitem{Tang2015} X. Tang, , W. Cheng, Y. Zhao, Variational principle for topological pressure on subsets, \emph{J. Math. Anal. Appl.} {\bf 424}(2015), 1272-1285.
	

	\bibitem{Thomspon2010} D. Thompson, The irregular set for maps with the specification property has full topological pressure, \emph{Dyn. Syst.} {\bf 25}(2010), 25-51.
	
	\bibitem{Thompson2012} D. Thompson, Irregular sets, the $\beta$-transformation and the almost specification property, \emph{Trans. Amer. Math. Soc. } {\bf 364}(2012), 5395-5414.
	
	\bibitem{Ures2012} R. Ures, Intrinsic ergodicity of partially hyperbolic diffeomorphisms with a hyperbolic linear part, \emph{Proc. Amer. Math. Soc.} {\bf 140}(2012), 1973-1985.
	
	
	\bibitem{Yang2020} D. Yang, J. Zhang, Non-hyperbolic ergodic measures and horseshoes in partially hyperbolic homoclinic classes, \emph{J. Inst. Math. Jussieu} {\bf 19}(2020), 1765-1792.
	
	
	\bibitem{Zel2021} A. Zelerowicz, Emergence for diffeomorphisms with nonzero Lyapunov exponents, arXiv:2111.07216.
	
	\bibitem{Zhao2018} C. Zhao, E. Chen, On the topological pressure of the saturated set with non-uniform structure, \emph{Topol. Methods Nonlinear Anal.} {\bf 51}(2018), 313-329.
	
	
	\bibitem{Zhou2013} X. Zhou, E. Chen, Multifractal analysis for the historic set in topological dynamical systems, \emph{Nonlineartiy} {\bf 26}(2013), 1975-1997.
	
\end{thebibliography}
\end{document}